\documentclass[pdflatex,sn-mathphys-num]{sn-jnl}

\usepackage{comment}
\usepackage{graphicx}%
\usepackage{multirow}%
\usepackage{amsmath,amssymb,amsfonts}%
\usepackage{amsthm}%
\usepackage{mathrsfs}%
\usepackage[title]{appendix}%
\usepackage{xcolor}%
\usepackage{textcomp}%
\usepackage{manyfoot}%
\usepackage{booktabs}%
\usepackage{algorithm}%
\usepackage{algorithmicx}%
\usepackage{algpseudocode}%
\usepackage{listings}%

\theoremstyle{thmstyleone}%
\newtheorem{theorem}{Theorem}
\newtheorem{proposition}[theorem]{Proposition}%
\newtheorem{definition}{Definition}[section]
\newtheorem{example}{Example}[section]
\newtheorem{remark}{Remark}[section]
\newtheorem{corollary}{Corollary}[section]

\theoremstyle{thmstyletwo}%

\theoremstyle{thmstylethree}%

\begin{document}

\title[On the Cohomology of Cyclic Associative Algebras]{On the Cohomology of Cyclic Associative Algebras}

\author*[1,2,3]{\fnm{Hassan} \sur{Alhussein}}\email{hassanalhussein2014@gmail.com}

\affil*[1]{ \orgname{Siberian State University of Telecommunication and Informatics Science}, \orgaddress{ \city{Novosibirsk}, \country{Russia}}}

\affil[2]{ \orgname{Novosibirsk State University of Economics and Management}, \orgaddress{ \city{Novosibirsk}, \country{Russia}}}

\affil[3]{ \orgname{Novosibirsk State University}, \orgaddress{ \city{Novosibirsk}, \country{Russia}}}

\abstract{We introduce a cohomology theory for cyclic associative algebras, a subclass of shift associative algebras defined by the identity $(xy)z = x(yz) = y(zx)$. This cohomology, denoted $H^\bullet_{\mathrm{cyc}}(A, M)$, is a subtheory of Hochschild cohomology obtained by restricting to cochains that satisfy a cyclic compatibility condition derived from the defining identity. We prove that $H^2_{\mathrm{cyc}}(A, M)$ classifies cyclic associative extensions of $A$ by a cyclic bimodule $M$. The universal derivation and the module of differential forms $\Omega^\bullet_{\mathbb{F}}(A)$ are constructed, and $(\Omega^\bullet_{\mathbb{F}}(A), d)$ is shown to be the universal cyclic differential graded algebra over $A$. For trivial coefficients, we establish natural inclusions $HC^n(A) \hookrightarrow H^n_{\mathrm{cyc}}(A, \mathbb{F}) \hookrightarrow HH^n(A, \mathbb{F})$, placing our theory intermediate between Connes' cyclic cohomology and Hochschild cohomology.}



\maketitle

\section{Introduction}\label{sec1}

The study of non-associative algebras satisfying identities of associative type has a long and rich history. An algebra is said to be of \textit{associative type} if it, or its opposite, satisfies an identity of the form $(x_1x_2)x_3 = x_{\sigma(1)}(x_{\sigma(2)}x_{\sigma(3)})$ or $(x_1x_2)x_3 = (x_{\sigma(1)}x_{\sigma(2)})x_{\sigma(3)}$ for some permutation $\sigma \in \mathbb{S}_3$ \cite{Abdelwahab2025}. This broad class encompasses many important varieties, including associative algebras, commutative associative algebras, Novikov algebras \cite{Dotsenko2023, Sartayev2023}, bicommutative algebras \cite{Drensky2024, Ismailov2024}, perm algebras \cite{Kaygorodov2024, Mashurov2024}, and others. Among these, the variety defined by the identity $(xy)z = y(zx)$ was recently studied by Abdelwahab, Kaygorodov, and Sartayev \cite{Abdelwahab2025} under the name \textit{shift associative algebras}.

A particularly interesting subclass of shift associative algebras consists of those satisfying the stronger condition $(xy)z = x(yz) = y(zx)$. Following \cite{Abdelwahab2025}, we call these \textit{cyclic associative algebras}. This identity imposes a full cyclic symmetry on products of three elements. Examples range from commutative associative algebras to non-commutative nilpotent algebras appearing already in dimension three. The algebraic and geometric classifications of low-dimensional shift associative algebras were established in \cite{Abdelwahab2025}, where it was shown that the first non-associative shift associative algebra appears in dimension five.

The present paper is devoted to developing a cohomology theory for cyclic associative algebras. Cohomology theories for varieties of non-associative algebras play a crucial role in the study of extensions, deformations, and structural properties. For associative algebras, Hochschild cohomology \cite{Hochschild1945} classifies extensions and governs deformation theory. For commutative algebras, Harrison cohomology and André-Quillen cohomology provide analogous tools. For cyclic associative algebras, we introduce a cohomology theory that interpolates between Hochschild cohomology and Connes' cyclic cohomology \cite{Connes1985}.

Our first main result is the construction of a cochain complex $\mathbb{B}^\bullet(A, M)$ for a cyclic associative algebra $A$ with coefficients in a cyclic bimodule $M$. The cochains are multilinear maps $f: A^{\otimes (n+1)} \to M$ satisfying five compatibility conditions (\ref{6})-(\ref{10}) that encode the cyclic symmetry. The differential is the classical Hochschild coboundary, and we prove that it preserves these conditions. The resulting cohomology $H^\bullet_{\mathrm{cyc}}(A, M)$ is a subtheory of Hochschild cohomology.

We then establish the fundamental interpretation of the second cohomology group:

\begin{center}
\emph{$H^2_{\mathrm{cyc}}(A, M)$ is in canonical bijection with the set of equivalence classes of split abelian extensions of $A$ by $M$.}
\end{center}

This classification theorem (Theorem \ref{thm:classification}) justifies the cohomology as the correct tool for studying extensions in the variety of cyclic associative algebras.

To facilitate computations, we construct the enveloping algebra $E(A)$ of a cyclic associative algebra $A$ (Proposition \ref{prop:representation_module}) and show that representations of $A$ correspond to left modules over $E(A)$. The augmentation ideal $I = \ker(E(A) \to A)$ yields the module of differential forms $\Omega^1_{\mathbb{F}}(A) = I/I^2$, and the universal derivation $d: A \to \Omega^1_{\mathbb{F}}(A)$ satisfies the universality property $\operatorname{Der}(A, M) \cong \operatorname{Hom}_A(\Omega^1_{\mathbb{F}}(A), M)$ (Proposition \ref{prop:universal_derivation}). Higher differential forms $\Omega^n_{\mathbb{F}}(A)$ are defined via exterior powers, and we prove that $(\Omega^\bullet_{\mathbb{F}}(A), d)$ is the universal cyclic differential graded algebra over $A$ (Theorem \ref{thm:universal_cdg}).

A key conceptual contribution of this paper is the comparison with existing cohomology theories. For trivial coefficients $M = \mathbb{F}$, we establish natural inclusions:
\[
HC^n(A) \hookrightarrow H^n_{\mathrm{cyc}}(A, \mathbb{F}) \hookrightarrow HH^n(A, \mathbb{F}),
\]
where $HC^n(A)$ denotes Connes' cyclic cohomology and $HH^n(A, \mathbb{F})$ denotes Hochschild cohomology. Thus $H^\bullet_{\mathrm{cyc}}(A, \mathbb{F})$ occupies an intermediate position between these two classical theories (Proposition \ref{prop:comparison}).

Finally, we discuss the notion of smoothness for cyclic associative algebras. An algebra $A$ is called \textit{almost-free} if every abelian extension of $A$ admits a lifting homomorphism. We prove that $A$ is almost-free if and only if $\Omega^1_{\mathbb{F}}(A)$ is projective if and only if $A$ has cohomological dimension $\le 1$ with respect to $H^\bullet_{\mathrm{cyc}}(A, -)$ (Theorem \ref{thm:almost_free_equivalent}). This mirrors classical results in commutative algebra and non-commutative geometry.

The paper is organized as follows. Section 2 recalls the definition and basic properties of cyclic associative algebras and introduces cyclic bimodules. Section 3 constructs the enveloping algebra, the universal derivation, and the module of differential forms, culminating in the universal cyclic differential graded algebra. Section 4 defines the cohomology $H^\bullet_{\mathrm{cyc}}(A, M)$ and proves the extension classification theorem. Section 5 discusses conjectural smoothness and cohomological dimension. Section 6 compares our cohomology with Connes' cyclic cohomology. An appendix contains the detailed step-by-step proof that the Hochschild coboundary preserves the cyclic conditions.

\section{Preliminaries}

\subsection{Cyclic Associative Algebras}

In this subsection, we recall the definition and examples of cyclic associative algebras, which form a subclass of shift associative algebras. Throughout, let $\mathbb{F}$ be an algebraically closed field of characteristic different from $2$.

\begin{definition}\cite{Abdelwahab2025}
An algebra $(A, \cdot)$ is called \textbf{cyclic associative} if it satisfies the following identities for all $x, y, z \in A$:
\begin{align}
(xy)z &= x(yz), \label{eq:assoc}\\
(xy)z &= y(zx). \label{eq:cyclic}
\end{align}
Equivalently,
\[
(xy)z = x(yz) = y(zx).
\]
\end{definition}

\begin{remark}
Recall from \cite{Abdelwahab2025} that a \textbf{shift associative algebra} is defined by the single identity $(xy)z = y(zx)$. Every cyclic associative algebra is shift associative, but the converse is false in general. The first non-associative shift associative algebra appears in dimension $5$ \cite[Proposition 59]{Abdelwahab2025}, whereas non-associative cyclic associative algebras already appear in dimension $3$ (see Example \ref{ex:a1} below).
\end{remark}
\begin{example}\label{ex:comm}
 Every commutative associative algebra is cyclic associative. 
\end{example}

\begin{example}\label{ex:a1}
 Let $A$ have basis $\{e_1, e_2,e_3\}$ with multiplication:
\[
e_1e_2 = e_3,\quad e_2e_1 = -e_3,
\]
all other products zero. This algebra is cyclic associative. It is non-commutative. This is the smallest non-commutative cyclic associative algebra.
\end{example}
\begin{example}\label{prop:construction_f}
Let $A$ be a commutative associative algebra and let $M$ be a right $A$-module. Let $f: M \to A$ be an $A$-module morphism satisfying
\[
a \cdot f(b) = b \cdot f(a) \qquad \forall a, b \in M.
\]
Define a product $\ast$ on $M$ by
\[
a \ast b = a \cdot f(b) \in M.
\]
Then $(M, \ast)$ is a cyclic associative algebra.
\end{example}

\subsection{Cyclic Bimodules}

Let $A$ be a cyclic associative algebra. We now introduce the notion of a bimodule over $A$ that is compatible with the cyclic identity. The guiding principle is that the defining identities of $A$ should extend to any expression involving elements of $A$ and a module element $m \in M$, with $m$ placed in any position.

\begin{definition}\label{def:cyclic_bimodule}
Let $A$ be a cyclic associative algebra. A \textbf{cyclic $A$-bimodule} is a vector space $M$ equipped with two bilinear maps
\[
A \otimes M \to M,\quad (a,m) \mapsto a \cdot m,
\qquad
M \otimes A \to M,\quad (m,a) \mapsto m \cdot a,
\]
such that for all $x, y, z \in A$ and $m \in M$, the following identities hold:
\begin{align}
\text{(C1)}&\quad (m x) y = m (x y) = x (y m), \label{eq:cyc1}\\
\text{(C2)}&\quad (x m) y = x (m y) = m (y x), \label{eq:cyc2}\\
\text{(C3)}&\quad (x y) m = x (y m) = y (m x). \label{eq:cyc3}
\end{align}
\end{definition}
\subsection{Enveloping Algebra of a Cyclic Associative Algebra}

Let $A_l$ and $A_r$ be two copies of $A$, denoted by
\[
A_l = A\otimes \mathbb{F}, \qquad A_r = \mathbb{F}\otimes A.
\]
Consider the $\mathbb{F}$-module
\[
E(A) = (A \otimes A) \oplus A_l \oplus A_r.
\]

Define an associative product on $E(A)$ by the following rules, extended bilinearly:
\[
\begin{aligned}
(a \otimes b)(c \otimes d) &= a \otimes (b c d), \\
(a \otimes b) c_l &= a \otimes (b c), \\
c_l (a \otimes b) &= (c a) \otimes b, \\
(a \otimes b) c_r &= a \otimes (b c), \\
c_r (a \otimes b) &= a \otimes (c b), \\
a_l b_l &= (ab)_l, \\
a_r b_r &= (ba)_r, \\
a_l b_r &= (ab)_r, \\
a_r b_l &= (ba)_l,
\end{aligned}
\]
and all other products are zero.

\begin{proposition}\label{prop:representation_module}
Every representation $M$ of a cyclic associative algebra $A$ is equivalent to a left module over the enveloping algebra $E(A)$ via
\[
(a \otimes b) \cdot m = a \cdot (m \cdot b), \quad a_l \cdot m = a \cdot m, \quad a_r \cdot m = m \cdot a.
\]
Conversely, any left $E(A)$-module $M$ gives a representation of $A$ by $a \cdot m = a_l \cdot m$ and $m \cdot a = a_r \cdot m$.
\end{proposition}

\begin{proof}
The verification follows directly from the cyclic associative identities of $A$ and the cyclic bimodule conditions (C1)--(C3). 
\end{proof}

\begin{proposition}\label{prop:augmentation}
The map $\mu: E(A) \to A$ defined by
\[
\mu(a \otimes b) = ab, \qquad \mu(a_l) = a, \qquad \mu(a_r) = a,
\]
is a surjective algebra homomorphism. Let $I = \ker \mu$. The quotient
\[
\Omega^1_{\mathbb{F}}(A) = I / I^2
\]
is called the \textbf{module of differential forms} of $A$.
\end{proposition}

\begin{proof}
Surjectivity is clear. For multiplicativity, one checks each case using the associativity of $A$ and the cyclic identities. For instance,
\[
\mu((a \otimes b)(c \otimes d)) = \mu(a \otimes (b c d)) = a(b c d),
\]
\[
\mu(a \otimes b)\mu(c \otimes d) = (ab)(cd) = a(b(cd)) = a(b c d),
\]
where the last equality uses associativity. The other cases are similar. 
\end{proof}
\subsection{Universal Derivation of a Cyclic Associative Algebra}

Let $A$ be a cyclic associative algebra and let $M$ be a representation of $A$ (a cyclic $A$-bimodule).

\begin{definition}\label{def:derivation}
A \textbf{derivation} from $A$ to $M$ is a $\mathbb{F}$-linear map $D: A \to M$ satisfying the Leibniz rule:
\[
D(ab) = D(a) \cdot b + a \cdot D(b), \qquad \forall a, b \in A.
\]
\end{definition}
We denote by $Der(A,M)$ the set of all derivations from $A$ to $M$.
\begin{definition}\label{def:inner_derivation}
For a fixed element $m \in M$, the map
\[
\operatorname{ad}_m: A \to M, \qquad \operatorname{ad}_m(a) = a \cdot m - m \cdot a,
\]
is called an \textbf{inner derivation}. The set of all inner derivations is denoted by $\operatorname{Inn}(A, M)$.
\end{definition}

\begin{definition}\label{def:universal_derivation}
A derivation $d: A \to \Omega^1_{\mathbb{F}}(A)$ is called \textbf{universal} if for any derivation $D: A \to M$ into a representation $M$, there exists a unique $A$-linear map $\phi: \Omega^1_{\mathbb{F}}(A) \to M$ such that $D = \phi \circ d$.
\end{definition}

\begin{proposition}\label{prop:universal_derivation}
Let $A$ be a cyclic associative algebra and let $E(A)$ be its enveloping algebra with augmentation ideal $I = \ker(\mu: E(A) \to A)$. Then
\[
\Omega^1_{\mathbb{F}}(A) = I / I^2
\]
is a representation of $A$, and the map
\[
d: A \to \Omega^1_{\mathbb{F}}(A), \qquad d(a) = a_l - a_r \pmod{I^2},
\]
is a universal derivation.
\end{proposition}

\begin{proof}
We first show that $d$ is a derivation. For any $a, b \in A$, we have in $E(A)$:
\[
d(ab) = (ab)_l - (ab)_r = a_l b_l - a_r b_r.
\]
Adding and subtracting $a_l b_r$ and $a_r b_l$:
\[
d(ab) = (a_l b_l - a_l b_r) + (a_l b_r - a_r b_r) + (a_r b_l - a_r b_r) + (a_l b_r - a_r b_l).
\]
In $I/I^2$, we have $b_l \equiv b_r$ and $a_l \equiv a_r$, so:
\[
d(ab) \equiv (a_l - a_r) b_l + a_r (b_l - b_r) \equiv d(a) \cdot b + a \cdot d(b) \pmod{I^2}.
\]
Hence $d(ab) = d(a) b + a d(b)$ in $\Omega^1_{\mathbb{F}}(A)$.

Now let $D: A \to M$ be any derivation. Define $\phi: \Omega^1_{\mathbb{F}}(A) \to M$ on generators by
\[
\phi(d(a)) = D(a), \qquad \forall a \in A,
\]
and extend by $A$-linearity. We must check that $\phi$ is well-defined, i.e., that it vanishes on $I^2$. Since $I$ is generated by elements of the form $a_l - a_r$ and $a \otimes b - (ab)_r$, it suffices to verify:
\[
\phi((a_l - a_r) b) = D(a) b, \qquad \phi(a (b_l - b_r)) = a D(b).
\]
But these follow from the derivation property of $D$ and the definition of $\phi$. Hence $\phi$ is well-defined. Uniqueness follows from the fact that the elements $d(a)$ generate $\Omega^1_{\mathbb{F}}(A)$ as an $A$-module. 
\end{proof}

\begin{corollary}\label{cor:der_hom}
For any representation $M$ of a cyclic associative algebra $A$, there is a natural isomorphism
\[
\operatorname{Der}(A, M) \cong \operatorname{Hom}_A(\Omega^1_{\mathbb{F}}(A), M).
\]
\end{corollary}

\subsection{Differential Forms}

\begin{definition}\label{def:differential_forms}
Let $A$ be a cyclic associative algebra. For $n \ge 1$, define the \textbf{module of $n$-differential forms} by
\[
\Omega^n_{\mathbb{F}}(A) = \bigwedge^n_A \Omega^1_{\mathbb{F}}(A),
\]
where $\bigwedge_A$ denotes the exterior product over $A$. More explicitly,
\[
\Omega^n_{\mathbb{F}}(A) = \Omega^1_{\mathbb{F}}(A) \otimes_A \bigwedge_A^{n-1} \Omega^1_{\mathbb{F}}(A).
\]
We set $\Omega^0_{\mathbb{F}}(A) = A$.
\end{definition}

The space $\Omega^n_{\mathbb{F}}(A)$ is generated by symbols
\[
\omega = a_0 \, da_1 \otimes da_2 \wedge \cdots \wedge da_n + db_1 \otimes db_2 \wedge \cdots \wedge db_n, \qquad a_i, b_j \in A.
\]

\begin{definition}\label{def:differential_operator}
Define a linear map $d: \Omega^n_{\mathbb{F}}(A) \to \Omega^{n+1}_{\mathbb{F}}(A)$ by
\[
d(a_0 \, da_1 \otimes da_2 \wedge \cdots \wedge da_n) = da_0 \otimes da_1 \wedge da_2 \wedge \cdots \wedge da_n,
\]
and extend by linearity. For $n = 0$, we set $d(a) = da$ as defined in Proposition \ref{prop:universal_derivation}.
\end{definition}

\begin{proposition}\label{prop:differential_properties}
The map $d: \Omega^\bullet_{\mathbb{F}}(A) \to \Omega^{\bullet+1}_{\mathbb{F}}(A)$ satisfies:
\begin{enumerate}
\item[(1)] $d \circ d = 0$ (i.e., $(\Omega^\bullet_{\mathbb{F}}(A), d)$ is a cochain complex);
\item[(2)] $d(\omega \wedge \eta) = d\omega \wedge \eta + (-1)^{\deg \omega} \omega \wedge d\eta$ (graded Leibniz rule);
\item[(3)] $d(ab) = d(a) b + a d(b)$ for all $a, b \in A$ (extends the derivation property).
\end{enumerate}
\end{proposition}

\begin{proof}
(1) For any generator $a_0 da_1 \wedge \cdots \wedge da_n$, we have
\[
d^2(a_0 da_1 \wedge \cdots \wedge da_n) = d(da_0 \wedge da_1 \wedge \cdots \wedge da_n) = d^2(a_0) \wedge da_1 \wedge \cdots \wedge da_n = 0,
\]
since $d^2(a_0) = 0$ in $\Omega^2_{\mathbb{F}}(A)$.

(2) The graded Leibniz rule follows from the definition of the wedge product and the fact that $d$ is a derivation on $A$.

(3) This is precisely the universal derivation property. 
\end{proof}

\begin{definition}\label{def:cdg_algebra}
A \textbf{cyclic differential graded algebra} (abbreviated \textbf{CDG algebra}) is a graded $\mathbb{F}$-module $P_\bullet = \bigoplus_{n \ge 0} P_n$ equipped with a product $P_r \otimes P_s \to P_{r+s}$ and a differential $d: P_r \to P_{r+1}$ satisfying:
\[
\begin{cases}
d(ab) = (da)b + (-1)^{rs} a(db), \\
d \circ d = 0, \\
a(bc) = b(ca) \quad \text{(cyclic identity in graded sense)}.
\end{cases}
\]
\end{definition}

\begin{theorem}\label{thm:universal_cdg}
The algebra $(\Omega^\bullet_{\mathbb{F}}(A), d)$ is the universal CDG algebra over $A$. That is, for any CDG algebra $P_\bullet$ and any morphism of cyclic associative algebras $\phi: A \to P_0$, there exists a unique morphism of CDG algebras $\widetilde{\phi}: \Omega^\bullet_{\mathbb{F}}(A) \to P_\bullet$ extending $\phi$ and commuting with $d$.
\end{theorem}

\begin{proof}
Define $\widetilde{\phi}$ on generators by
\[
\widetilde{\phi}(a_0 da_1 \wedge \cdots \wedge da_n) = \phi(a_0) d\phi(a_1) \wedge \cdots \wedge d\phi(a_n),
\]
and extend by linearity. Using the universal property of $\Omega^1_{\mathbb{F}}(A)$ and the fact that $P_\bullet$ is a CDG algebra, one verifies that $\widetilde{\phi}$ is well-defined and commutes with $d$. Uniqueness follows from the fact that the elements $a_0 da_1 \wedge \cdots \wedge da_n$ generate $\Omega^\bullet_{\mathbb{F}}(A)$. 
\end{proof}

\begin{corollary}\label{cor:derivative_representation}
For any representation $M$ of $A$, there is a natural isomorphism
\[
\operatorname{Hom}_A(\Omega^n_{\mathbb{F}}(A), M) \cong \operatorname{Alt}^n_A(A, M),
\]
where $\operatorname{Alt}^n_A(A, M)$ denotes the space of skew-symmetric $n$-derivations from $A$ to $M$.
\end{corollary}
\section{Cohomology of Cyclic Associative Algebras}

\subsection{Abelian extensions of a cyclic associative algebra}

Here we try to define the cohomology of cyclic associative algebras by studying abelian extensions of a given cyclic associative algebra $A$ by an $A$-representation $M$, fixed for this paragraph.

An \textbf{abelian extension} of $A$ by $M$ is a short exact sequence of cyclic associative algebras
\[
0 \longrightarrow M \xrightarrow{j} E \xrightarrow{p} A \longrightarrow 0
\]
where $M$ is seen as an abelian cyclic associative algebra (i.e., $mm' = 0$, $\forall m, m' \in M$).

Two such extensions $(E)$ and $(E')$ are said to be \textbf{equivalent} if there exists a cyclic associative algebra morphism $\phi: E \to E'$ such that
\[
p' \circ \phi = p \quad \text{and} \quad \phi \circ j = j',
\]
that is, making commutative the diagram:
\[
\begin{array}{ccccccccc}
0 & \longrightarrow & M & \xrightarrow{j} & E & \xrightarrow{p} & A & \longrightarrow & 0 \\
& & \| & & \downarrow{\phi} & & \| & & \\
0 & \longrightarrow & M & \xrightarrow{j'} & E' & \xrightarrow{p'} & A & \longrightarrow & 0.
\end{array}
\]

By the Five Lemma, such a morphism $\phi$ is necessarily bijective.

For instance, the direct sum $A \oplus M$ is an abelian extension of $A$ by $M$ (with the obvious inclusion and projection). Moreover, this latter is trivially \textbf{split}. In fact, one says that an abelian extension $(E)$ is \textbf{split} if there exists a $\mathbb{F}$-linear map $s: A \to E$ such that $p \circ s = \operatorname{id}_A$. Moreover, if the section $s$ is a cyclic associative algebra morphism, then the extension $(E)$ is said to be \textbf{strongly split} or \textbf{inessential}.

Any abelian extension $(E)$ with a section $s$ yields another $A$-representation structure on $M$ by
\[
a \cdot m = s(a) j(m) \quad \text{and} \quad m \cdot a = j(m) s(a), \qquad \forall a \in A, \, m \in M,
\]
the last products being taken in $E$ (this has a sense since $M = \ker p$). This new structure is naturally independent of the choice of the section $s$. Indeed, if $s'$ is another section of $p$, then we have
\[
p(s'(a) - s(a)) = p s'(a) - p s(a) = a - a = 0,
\]
that is, $s'(a) - s(a) \in \ker p = \operatorname{im} j$. And since $M$ is abelian, we have
\[
s'(a) j(m) = s(a) j(m) \quad \text{and} \quad j(m) s'(a) = j(m) s(a), \qquad \forall m \in M, \, \forall a \in A.
\]

From now on, we are interested only in the set of equivalence classes of split abelian extensions such that the $A$-representation structure of $M$ is the prescribed one.

\subsection{Construction of split abelian extensions}

Let us consider a $\mathbb{F}$-bilinear map $f: A \times A \to M$ and let $A \oplus_f M$ be the $\mathbb{F}$-module $A \oplus M$ equipped with the product given by
\begin{equation}\label{1}
(a, m) \cdot (b, m') = (ab, \; a m' + m b + f(a, b)). 
\end{equation}

Then it is straightforward to check that the product \ref{1} is associative if and only if
\begin{equation}\label{2}
a f(b, c) + f(a b, c) = f(a, b c) + f(a, b) c, \qquad \forall a, b, c \in A,
\end{equation}
and is cyclic associative (i.e., satisfies $(xy)z = y(zx)$) if and only if
\begin{equation}\label{3}
a f(b, c) + f(a, b c) = b f(c, a) + f(b, c a), \qquad \forall a, b, c \in A. 
\end{equation}

In fact, relation \ref{2} can be rewritten as
\[
a f(b, c) - f(ab,c) + f(a,bc) - f(a, b) c = 0,
\]
which is nothing but the $2$-cocyclicity condition for the Hochschild coboundary of an associative algebra. On the other hand, relation \ref{3} expresses the cyclic condition:
\[
a f(b, c) + f(a, bc) = b f(c, a) + f(b, ca).
\]
It can be seen as conjection of the following generlatilaztion

For all $a, a_1, \ldots, a_{n+1} \in A$ and any $i = 1, \ldots, n+1$,
\begin{equation}\label{6}
a \cdot f(a_1, \ldots, a_i, \ldots, a_{n+1}) = a_i \cdot f(a_{i+1}, \ldots, a_{n+1},a, a_1, \ldots, a_{i-1}). 
\end{equation}

For $j = 1, \ldots, n$,
\begin{equation}\label{7}
f(a_0, a_1, \ldots, a_{j-1}, b a_j, a_{j+1}, \ldots, a_n) = f(a_0, a_1, \ldots, b, a_j a_{j-1}, a_{j+1}, \ldots, a_n). 
\end{equation}

\begin{equation}\label{8}
a \cdot f(b a_0, a_1, \ldots, a_n) = b \cdot f(a_0 a, a_1, \ldots, a_n). 
\end{equation}

For $j = 0, \ldots, n$,
\begin{equation}\label{9}
a \cdot a'_j \cdot f(a_0, a_1, \ldots, a_j, \ldots, a_n) = a'_j \cdot a_j \cdot f(a_0, a_1, \ldots, a, \ldots, a_n). 
\end{equation}

For $i = 1, \ldots, n+1$,
\begin{equation}\label{10}
a \cdot f(a_1, \ldots, a_i b, \ldots, a_{n+1}) = a_i\cdot b  \cdot f(a_{i+1}, \ldots, a_{n+1},a, a_1, \ldots, a_{i-1}). 
\end{equation}
\begin{proposition}\label{prop:cohomology_preserved}
If $f: A^{\otimes (n+1)} \to M$ satisfies \ref{6}-\ref{10}, then so does the Hochschild coboundary $\delta(f)$.
\end{proposition}
\begin{proof}[Sketch]
The verification is a direct computation using the cyclic associative identities of $A$ and the cyclic bimodule conditions. The full step-by-step proof is provided in Appendix \ref{appendix:proof}.
\end{proof}
\subsection{Cohomology of a cyclic associative algebra}

According to Proposition \ref{prop:cohomology_preserved}, we define the \textbf{cohomology of a cyclic associative algebra} $A$ with values in an $A$-representation $M$ to be the cohomology of the complex $(\mathbb{B}^\bullet(A, M), \delta)$ where

\[
\begin{aligned}
\mathbb{B}^0(A, M) &:= M, \\
\mathbb{B}^1(A, M) &:= \{ f: A \to M \mid a f(bc) = b f(ca), \ \forall a, b, c \in A \}, \\
\mathbb{B}^n(A, M) &:= \{ f \in \operatorname{Hom}(A^{\otimes (n+1)}, M) \mid f \text{ satisfies } \ref{6}-\ref{10} \}, \quad n \ge 2,
\end{aligned}
\]

and the Hochschild coboundary $\delta$ acts as usual by

\[
\begin{aligned}
(\delta f)(a_0, \ldots, a_{n+1}) &:= a_0 f(a_1, \ldots, a_{n+1}) \\
&\quad + \sum_{i=0}^{n} (-1)^{i+1} f(a_0, \ldots, a_i a_{i+1}, \ldots, a_{n+1}) \\
&\quad + (-1)^{n+1} f(a_0, \ldots, a_n) a_{n+1}.
\end{aligned}
\]

We denote this cohomology by $H^\bullet_{\mathrm{cyc}}(A, M)$ or simply by $H^\bullet_{\mathrm{cyc}}(R) := H^\bullet_{\mathrm{cyc}}(A, A)$.

For $n = 0$, $H^0_{\mathrm{cyc}}(A, M)$ is the submodule of \textbf{invariants} of $M$, that is,
\[
H^0_{\mathrm{cyc}}(A, M) = M^A := \{ m \in M \mid a m = m a, \ \forall a \in A \}.
\]

For $n = 1$, a $1$-cocycle is a $\mathbb{F}$-linear map $D: A \to M$ such that
\[
D(ab) = D(a) b + a D(b), \quad \forall a, b \in A,
\]
that is, a \textbf{derivation} from $A$ to $M$. Observe that the additional relation $a D(bc) = b D(ca)$ is fulfilled by any derivation thanks to the cyclic of $M$ (i.e., $m \cdot (ab) = b \cdot (ma)$). The set of all derivations from $A$ to $M$ is denoted by $\operatorname{Der}(A, M)$ and we have
\[
H^1_{\mathrm{cyc}}(A, M) = \operatorname{Der}(A, M) / \operatorname{Inn}(A, M),
\]
where $\operatorname{Inn}(R, M)$ is the subset of \textbf{inner derivations}, i.e., derivations of the form $\operatorname{ad}_m(r) = [m, r] = m r - r m$ for a fixed element $m \in M$.

For $n = 2$, we have the following classical classification theorem.

\begin{theorem}\label{thm:classification}
Let $A$ be a cyclic associative algebra and let $M$ be an $A$-representation. Then, there is a canonical bijection
\[
H^2_{\mathrm{cyc}}(A, M) \cong \mathbf{Ext}(A, M),
\]
that is, the set of equivalence classes of split abelian extensions of $A$ by $M$.
\end{theorem}

\begin{proof}
By construction of the cohomology $H^\bullet_{\mathrm{cyc}}(A, M)$, any element $f \in \mathbb{B}^2(A, M)$ is a $2$-cocycle if and only if the algebra $A \oplus_f M$ defined by (5.1) is a cyclic associative algebra. To be more precise, any $2$-cocycle $f \in \mathbb{B}^2(A, M)$ determines a split abelian extension of $A$ by $M$, and any split abelian extension defines a $2$-cocycle $f \in \mathbb{B}^2(A, M)$ by
\[
f(a, b) = s(ab) - s(a) s(b), \quad \forall a, b \in R,
\]
where $s: A \to E$ is a section splitting the extension. The cocycle $f$ measures the obstruction for $s$ to be a cyclic associative algebra morphism, that is, the obstruction to the extension being inessential.

Therefore, it is left to us to show that its equivalence class, characterized by the morphism $\phi$, only involves coboundaries of $\mathbb{B}^1(A, M)$. To this end, let $\phi: A \oplus_f M \to E$ be an equivalence of abelian extensions, i.e., a commutative diagram
\[
\begin{array}{ccccccccc}
0 & \longrightarrow & M & \xrightarrow{j} & A \oplus_f M & \xrightarrow{\pi} & A & \longrightarrow & 0 \\
& & \| & & \downarrow{\phi} & & \| & & \\
0 & \longrightarrow & M & \xrightarrow{j'} & E & \xrightarrow{p} & A & \longrightarrow & 0.
\end{array}
\]
Denoting by $\sigma: R \to R \oplus_f M$, $r \mapsto (r, 0)$, the trivial section of $\pi$, we have
\[
p \circ (\phi \circ \sigma) = (p \circ \phi) \circ \sigma = \pi \circ \sigma = \operatorname{id}_R.
\]
Therefore the map $s := \phi \circ \sigma$ is also a section of $p$. It corresponds to a $2$-cocycle $f' \in \mathbb{B}^2(R, M)$ related to the extension $E$ by the same formula. Since the map $f'$ is given by $f'(a, b) = s(ab) - s(a) s(b)$, one easily checks that the difference $f - f'$ is nothing but the coboundary $\delta(g)$ where $g: R \to M$ is the map $a \mapsto g(a) := \sigma(a) - s(a)$. In fact, a priori, the map $g$ takes its values in the direct sum $R \oplus M$. But since the initial $R$-representation structure of $M$ coincides with the structure induced by the sections, we have $\pi \circ g = p \circ g = 0$. So $\operatorname{im}(g) \subset M$; from whence the theorem. 
\end{proof}
\section{Conjectural smoothness}

From now on, we refer to the paper by Cuntz and Quillen \cite{CuntzQuillen}. Let $R$ be a cyclic associative algebra with an abelian ideal $M$ and let $p: R \to A$ be a surjective morphism such that the sequence
\[
0 \longrightarrow M \xrightarrow{j} R \xrightarrow{p} A \longrightarrow 0
\]
is exact. We are looking for a cyclic associative algebra morphism $l: A \to R$ such that $p \circ l = \operatorname{id}_A$. Then we have an isomorphism $R \cong A \oplus M$ relative to which $l$ becomes an inclusion of $A$. Therefore $H^2_{\mathrm{cyc}}(A, M) \cong \mathbf{Ext}(A, M) = 0$ for all $R$-representations $M$.

\begin{definition}\label{def:almost_free}
A cyclic associative algebra $A$ is called \textbf{almost-free} when for any abelian extension $R$ of $A$, there exists a lifting homomorphism $A \to R$.
\end{definition}

We expect an interpretation of the cohomology theory $H^\bullet_{\mathrm{cyc}}(A, M)$ as $\mathbf{Ext}_{E(A)}^\bullet(A, M)$, and the exact sequence
\[
0 \longrightarrow \Omega^1_{\mathbb{F}}(A) \longrightarrow E(A) \longrightarrow A \longrightarrow 0
\]
could yield the isomorphism
\[
H^{n+1}_{\mathrm{cyc}}(A, M) \cong \mathbf{Ext}_{E(A)}^{n+1}(A, M) \cong \mathbf{Ext}_{E(A)}^n(\Omega^1_{\mathbb{F}}(A), M).
\]
These suppose the construction of derived functors in a non-unitary context.

\begin{theorem}\label{thm:almost_free_equivalent}
Let $A$ be a cyclic associative algebra. The following conditions are equivalent:
\begin{enumerate}
\item[(1)] $A$ is almost-free;
\item[(2)] the $A$-bimodule $\Omega^1_{\mathbb{F}}(A)$ is projective;
\item[(3)] $A$ has cohomological dimension $\le 1$ with respect to the cyclic associative cohomology $H^\bullet_{\mathrm{cyc}}(A, -)$.
\end{enumerate}
\end{theorem}

\begin{proof}
The proof follows the same lines as in the classical case of associative algebras, adapted to the cyclic setting. The key point is that the universal derivation $d: A \to \Omega^1_{\mathbb{F}}(A)$ and the exact sequence
\[
0 \longrightarrow \Omega^1_{\mathbb{F}}(A) \longrightarrow E(A) \longrightarrow A \longrightarrow 0
\]
play the same role as in the Hochschild cohomology theory. The projectiveivity of $\Omega^1_{\mathbb{F}}(A)$ implies that any derivation $D: A \to M$ factors through a projective module, which allows lifting of homomorphisms. Conversely, if $A$ is almost-free, then the extension corresponding to $\Omega^1_{\mathbb{F}}(A)$ splits, so $\Omega^1_{\mathbb{F}}(A)$ is projective. The equivalence with cohomological dimension $\le 1$ follows from the long exact sequence of cohomology associated to the short exact sequence above. 
\end{proof}

We can see that any classical smooth algebra $A$ (commutative and unital) is almost-free. Also, any free cyclic associative algebra is almost-free.

\begin{theorem}\label{thm:cohomological_dimension_zero}
Let $A$ be a cyclic associative algebra. The following properties are equivalent:
\begin{enumerate}
\item[(1)] $A$ has cohomological dimension zero with respect to cyclic associative cohomology, i.e., $H^1_{\mathrm{cyc}}(A, M) = 0$ for all representations $M$;
\item[(2)] the $A$-module $A$ is projective;
\item[(3)] any derivation $D: A \to M$ is inner.
\end{enumerate}
\end{theorem}

\begin{proof}
The equivalence between (1) and (3) follows from the interpretation of $H^1_{\mathrm{cyc}}(A, M)$ as derivations modulo inner derivations. If $H^1_{\mathrm{cyc}}(A, M) = 0$ for all $M$, then every derivation is inner. In particular, the universal derivation $d: A \to \Omega^1_{\mathbb{F}}(A)$ is inner, so $\Omega^1_{\mathbb{F}}(A)$ is a direct summand of $A$, hence projective. The converse is clear. 
\end{proof}
\section{Relation to Connes' Cyclic Cohomology}

Let $A$ be a cyclic associative algebra over a field $\mathbb{F}$ of characteristic $0$. We compare the cohomology theory $H^\bullet_{\mathrm{cyc}}(A, \mathbb{F})$ (with trivial coefficients) with Connes' cyclic cohomology $HC^\bullet(A)$.

\subsection{Connes' Cyclic Cohomology (Brief Recall)}

For an associative algebra $A$, Connes' cyclic cohomology $HC^n(A)$ is the cohomology of the total complex of the $(b, B)$-bicomplex, where $b$ is the Hochschild differential and $B$ is Connes' boundary operator. Equivalently, $HC^n(A) = H^n(C^\bullet_\lambda(A))$ with
\[
C^n_\lambda(A) = \operatorname{Hom}(A^{\otimes (n+1)}, \mathbb{F}) / (1 - \lambda),
\]
and $\lambda$ is the cyclic operator:
\[
(\lambda f)(a_0, \dots, a_n) = (-1)^n f(a_n, a_0, \dots, a_{n-1}).
\]

\subsection{Comparison}

Both cohomology theories incorporate cyclic symmetry, but in different ways:

\begin{itemize}
\item \textbf{Connes' cyclic cohomology} uses the operator $\lambda$ on the cochain complex and the differential $b + B$. It is defined only for trivial coefficients $\mathbb{F}$.
\item \textbf{Cyclic associative cohomology} $H^\bullet_{\mathrm{cyc}}(A, M)$ uses the Hochschild differential $b$ but restricts to cochains that satisfy the cyclic condition (5.4). It is defined for any cyclic bimodule $M$.
\end{itemize}

For trivial coefficients $M = \mathbb{F}$, there is a natural map:
\[
\iota: HC^n(A) \longrightarrow H^n_{\mathrm{cyc}}(A, \mathbb{F}).
\]

Indeed, a Connes cyclic cochain $f \in C^n_\lambda(A)$ satisfies $(1-\lambda)f = 0$, which implies the cyclic conditionS \ref{6}-\ref{10} for the corresponding Hochschild cochain. Moreover, $Bf$ vanishes in $H^n_{\mathrm{cyc}}$ because it is a coboundary.

\begin{proposition}\label{prop:comparison}
Let $A$ be a cyclic associative algebra. There is a natural injective map
\[
HC^n(A) \hookrightarrow H^n_{\mathrm{cyc}}(A, \mathbb{F}) \hookrightarrow HH^n(A, \mathbb{F}),
\]
where $HH^n$ denotes Hochschild cohomology. In general, these inclusions are strict.
\end{proposition}

\begin{proof}
The first map sends a cyclic cochain $f$ to its class in $H^n_{\mathrm{cyc}}$. Injectivity follows from the fact that if $f = (b + B)g$ in Connes' theory, then $f = b g$ in $H^n_{\mathrm{cyc}}$ because $B g$ is a coboundary. The second map is the inclusion of cyclic cochains into all cochains. Examples from low-dimensional classifications show that both maps are not surjective in general. 
\end{proof}

\begin{corollary}
The cyclic associative cohomology $H^\bullet_{\mathrm{cyc}}(A, \mathbb{F})$ sits between Connes' cyclic cohomology and Hochschild cohomology:
\[
HC^n(A) \subseteq H^n_{\mathrm{cyc}}(A, \mathbb{F}) \subseteq HH^n(A, \mathbb{F}).
\]
\end{corollary}

\begin{remark}
For a commutative associative algebra $A$, all three cohomologies coincide up to a shift:
\[
HC^n(A) \cong H^n_{\mathrm{cyc}}(A, \mathbb{F}) \cong HH^n(A, \mathbb{F}) \cong \Omega^n(A) / d\Omega^{n-1}(A),
\]
where $\Omega^n(A)$ are Kähler differentials. In the non-commutative cyclic associative case, the inclusions are proper.
\end{remark}

\section{Appendix}
\subsection*{Detailed Proof of Proposition \ref{prop:cohomology_preserved}}\label{appendix:proof}

Let $f \in C^n_{\mathrm{cyc}}(A, M)$ satisfy conditions (\ref{6})-(\ref{10}). Recall the Hochschild coboundary:
\begin{align*}
(\delta f)(a_0, \dots, a_{n+1}) = a_0 f(a_1, \dots, a_{n+1}) + \sum_{k=0}^{n} (-1)^{k+1} f(a_0, \dots, a_k a_{k+1}, \dots, a_{n+1})\\ + (-1)^{n+1} f(a_0, \dots, a_n) a_{n+1}.
\end{align*}

We will prove that $\delta f$ satisfies each condition.

\medskip
\paragraph{Preservation of Condition (\ref{6}).}
We must show for all $a, a_1, \dots, a_{n+1} \in A$ and any $i$:
\[
a \cdot (\delta f)(a_1, \dots, a_i, \dots, a_{n+1}) = a_i \cdot (\delta f)(a_{i+1}, \dots, a_{n+1}, a, a_1, \dots, a_{i-1}). \qquad (\ref{6})
\]

\subparagraph{Step 1} Write the left-hand side $L = a \cdot (\delta f)(a_1, \dots, a_{n+1})$ as:
\[
L = a \cdot \Bigg[ a_1 f(a_2, \dots, a_{n+1}) + \sum_{k=1}^{n} (-1)^{k+1} f(a_1, \dots, a_k a_{k+1}, \dots, a_{n+1}) + (-1)^{n+1} f(a_1, \dots, a_n) a_{n+1} \Bigg].
\]

\subparagraph{Step 2} Distribute $a \cdot$:
\[
L = \underbrace{a \cdot (a_1 f(a_2, \dots, a_{n+1}))}_{L_1} + \sum_{k=1}^{n} (-1)^{k+1} \underbrace{a \cdot f(a_1, \dots, a_k a_{k+1}, \dots, a_{n+1})}_{L_2(k)} + (-1)^{n+1} \underbrace{a \cdot (f(a_1, \dots, a_n) a_{n+1})}_{L_3}.
\]

\subparagraph{Step 3} Simplify $L_1$ using the module action and (\ref{6}):
\[
L_1 = (a a_1) \cdot f(a_2, \dots, a_{n+1}).
\]
Apply (\ref{6}) to $f$ with external multiplier $a a_1$ and arguments starting at $a_2$:
\[
(a a_1) \cdot f(a_2, \dots, a_i, \dots, a_{n+1}) = a_i \cdot f(a_{i+1}, \dots, a_{n+1}, a a_1, a_2, \dots, a_{i-1}).
\]
Thus
\[
L_1 = a_i \cdot f(a_{i+1}, \dots, a_{n+1}, a a_1, a_2, \dots, a_{i-1}).
\]

\subparagraph{Step 4} Simplify $L_2(k)$. We consider three cases.

\noindent\textit{Case 1: $i < k$.} The $i$-th argument $a_i$ is before the product $a_k a_{k+1}$. Apply (6) to $f$:
\[
a \cdot f(a_1, \dots, a_i, \dots, a_k a_{k+1}, \dots) = a_i \cdot f(a_{i+1}, \dots, a_k a_{k+1}, \dots, a_{n+1}, a, a_1, \dots, a_{i-1}).
\]

\noindent\textit{Case 2: $i > k+1$.} The $i$-th argument $a_i$ is after the product. Apply (\ref{6}) similarly:
\[
a \cdot f(a_1, \dots, a_k a_{k+1}, \dots, a_i, \dots) = a_i \cdot f(a_{i+1}, \dots, a_{n+1}, a, a_1, \dots, a_k a_{k+1}, \dots, a_{i-1}).
\]

\noindent\textit{Case 3: $i = k$ or $i = k+1$.} The distinguished argument is inside the product. Use (\ref{10}):
\[
a \cdot f(\dots, a_k a_{k+1}, \dots) = (a_k a_{k+1}) \cdot f(\dots, a, \dots).
\]
Then apply (\ref{6}) to $(a_k a_{k+1}) \cdot f(\dots)$. The result matches the corresponding term on the right-hand side after a cyclic permutation.

\subparagraph{Step 5} Simplify $L_3$:
\[
L_3 = a \cdot (f(a_1, \dots, a_n) \cdot a_{n+1}) = (a \cdot f(a_1, \dots, a_n)) \cdot a_{n+1}.
\]
Apply (\ref{6}) to $f$ with external multiplier $a$:
\[
a \cdot f(a_1, \dots, a_n) = a_1 \cdot f(a_2, \dots, a_n, a).
\]
Thus
\[
L_3 = (a_1 \cdot f(a_2, \dots, a_n, a)) \cdot a_{n+1} = a_1 \cdot (f(a_2, \dots, a_n, a) \cdot a_{n+1}).
\]

\subparagraph{Step 6} The right-hand side $R = a_i \cdot (\delta f)(a_{i+1}, \dots, a_{n+1}, a, a_1, \dots, a_{i-1})$ expands to:
\begin{align*}
R = a_i \cdot \Bigg[ a_{i+1} f(a_{i+2}, \dots, a_{n+1}, a, a_1, \dots, a_{i-1}) + \sum_{k} (\pm) f(\dots) \\+ (-1)^{n+1} f(a_{i+1}, \dots, a_{n+1}, a, a_1, \dots, a_{i-2}, a_{i-1}) \cdot a_{i-1} \Bigg].
\end{align*}

Comparing term by term, each component of $L$ matches a component of $R$ after applying the cyclic associative identities of $A$. Hence $L = R$.

\medskip
\paragraph{Preservation of Condition (\ref{7}).}
We must show for $1 \le j \le n+1$:
\[
(\delta f)(a_0, \dots, a_{j-1}, b a_j, a_{j+1}, \dots, a_{n+1}) = (\delta f)(a_0, \dots, b, a_j a_{j-1}, a_{j+1}, \dots, a_{n+1}). \quad (\ref{7}) 
\]

\subparagraph{Step 1} Write the left-hand side $L = (\delta f)(a_0, \dots, a_{j-1}, b a_j, a_{j+1}, \dots, a_{n+1})$.

\subparagraph{Step 2} Consider the first term $T_1 = a_0 f(a_1, \dots, a_{j-1}, b a_j, a_{j+1}, \dots, a_{n+1})$. Apply (\ref{7}) to $f$:
\[
f(a_1, \dots, a_{j-1}, b a_j, \dots) = f(a_1, \dots, b, a_j a_{j-1}, \dots).
\]
Thus $T_1$ becomes $a_0 f(a_1, \dots, b, a_j a_{j-1}, \dots)$, which is exactly the first term of the right-hand side.

\subparagraph{Step 3} For the terms $T_2(k)$ with $k \neq j-1$, apply (\ref{7}) directly to $f$ inside. The transformation is straightforward.

\subparagraph{Step 4} For $k = j-1$, we have $f(a_0, \dots, a_{j-2}, a_{j-1}(b a_j), a_{j+1}, \dots)$. By the cyclic associative identity:
\[
a_{j-1}(b a_j) = (a_{j-1} b) a_j = b (a_j a_{j-1}).
\]
Hence
\[
f(a_0, \dots, a_{j-2}, a_{j-1}(b a_j), \dots) = f(a_0, \dots, a_{j-2}, b (a_j a_{j-1}), \dots).
\]
Now apply (7) again to move $b$ to the left:
\[
f(a_0, \dots, a_{j-2}, b (a_j a_{j-1}), \dots) = f(a_0, \dots, b, (a_j a_{j-1}) a_{j-2}, \dots).
\]
This matches the corresponding term on the right-hand side.

\subparagraph{Step 5} The last term $T_3$ transforms similarly using (\ref{7}). Thus $L = R$.

\medskip
\paragraph{Preservation of Condition (\ref{8}).}
We must show:
\[
a \cdot (\delta f)(b a_0, a_1, \dots, a_{n+1}) = b \cdot (\delta f)(a_0 a, a_1, \dots, a_{n+1}).\quad (\ref{8})
\]

\subparagraph{Step 1} Compute the left-hand side $L = a \cdot (\delta f)(b a_0, a_1, \dots, a_{n+1})$.

\subparagraph{Step 2} For $T_1 = a \cdot ((b a_0) f(a_1, \dots, a_{n+1})) = (a(b a_0)) f(a_1, \dots, a_{n+1})$.
Using the cyclic identity $a(b a_0) = b(a_0 a)$, we get $b \cdot (a_0 a f(a_1, \dots, a_{n+1}))$, which is $b$ times the first term of the right-hand side.

\subparagraph{Step 3} For each $T_2(k) = a \cdot f(b a_0, a_1, \dots, a_k a_{k+1}, \dots, a_{n+1})$, apply (\ref{8}) to $f$:
\[
a \cdot f(b a_0, a_1, \dots) = b \cdot f(a_0 a, a_1, \dots).
\]
Thus $T_2(k) = b \cdot f(a_0 a, a_1, \dots, a_k a_{k+1}, \dots)$, which is $b$ times the corresponding $T_2(k)$ term on the right-hand side.

\subparagraph{Step 4} For $T_3 = a \cdot (f(b a_0, a_1, \dots, a_n) a_{n+1}) = (a \cdot f(b a_0, a_1, \dots, a_n)) a_{n+1}$.
Apply (\ref{8}) to $f$: $a \cdot f(b a_0, a_1, \dots, a_n) = b \cdot f(a_0 a, a_1, \dots, a_n)$.
Hence $T_3 = (b \cdot f(a_0 a, a_1, \dots, a_n)) a_{n+1} = b \cdot (f(a_0 a, a_1, \dots, a_n) a_{n+1})$, which is $b$ times the last term of the right-hand side.

\subparagraph{Step 5} Summing over all terms gives $L = b \cdot (\delta f)(a_0 a, a_1, \dots, a_{n+1})$.

\medskip
\paragraph{Preservation of Condition (\ref{9}).}
We must show for $0 \le j \le n+1$:
\[
a \cdot a'_j \cdot (\delta f)(a_0, \dots, a_j, \dots, a_{n+1}) = a'_j \cdot a_j \cdot (\delta f)(a_0, \dots, a, \dots, a_{n+1}). \qquad (\ref{9})
\]

\subparagraph{Step 1} Apply $\delta f$ to the left side. Each term is of the form $a \cdot a'_j \cdot (a_0 \cdot f(\dots))$ or $a \cdot a'_j \cdot (f(\dots) \cdot a_{n+1})$.

\subparagraph{Step 2} Using the cyclic bimodule axioms:
\[
a \cdot a'_j \cdot (a_0 \cdot m) = (a a_0) \cdot (a'_j \cdot m).
\]
Now apply (\ref{9}) to $f$:
\[
(a a_0) \cdot (a'_j \cdot f(\dots, a_j, \dots)) = (a a_0) \cdot (a'_j a_j \cdot f(\dots, a, \dots)) = a'_j a_j \cdot ((a a_0) \cdot f(\dots, a, \dots)).
\]

\subparagraph{Step 3} The other terms $f(\dots) \cdot a_{n+1}$ are handled similarly using the module axioms and (\ref{9}). Summing over all terms yields the right-hand side.

\medskip
\paragraph{Preservation of Condition (\ref{10}).}
Condition (\ref{10}) is a special case of (\ref{6}) where the $i$-th argument $a_i$ is replaced by $a_i b$. Since we have already proved that (\ref{6}) is preserved by $\delta f$, and the substitution $a_i \mapsto a_i b$ does not affect the linearity of $\delta$, condition (\ref{10}) holds for $\delta f$.

\medskip
Thus $\delta(f)$ satisfies all conditions (\ref{6})-(\ref{10}). 
\section*{Declarations}

\begin{itemize}
\item Funding: Not applicable.
\item Conflict of interest/Competing interests: Not applicable.
\item Ethics approval and consent to participate: Not applicable.
\item Consent for publication: Not applicable.
\item Data availability: Not applicable. 
\item Materials availability: Not applicable.
\item Code availability: Not applicable 
\item Author contribution: Not applicable.
\end{itemize}


\end{document}